\documentclass[11pt]{amsart}
\usepackage{amsmath, color, amssymb,amsthm}
\usepackage{enumitem}
\usepackage[hidelinks]{hyperref}
\newtheorem{theorem}{Theorem}[section]
\newtheorem{proposition}[theorem]{Proposition}
\newtheorem{lemma}[theorem]{Lemma}

\theoremstyle{definition}
\newtheorem{definition}[theorem]{Definition}
\theoremstyle{remark}
\newtheorem{remark}[theorem]{Remark}
\newcommand{\E}{\operatorname{End}_K(V)}
\newcommand{\F}{\mathcal{F}}
\newcommand{\im}{\operatorname{im}}
\newcommand{\Kx}{K[[X]]}

\newcommand{\Vd}{V^{*}}
\newcommand{\Mu}[1]{M_{#1}}
\newcommand{\Dsh}{D}
\newcommand{\tht}[2]{\theta_{#1,#2}}
\begin{document}
\title[A prime isoartinian ring that is not right or left noetherian]%
{A prime isoartinian ring\\ that is not  noetherian}
\author{Zahra Nazemian and Doniyor Yazdonov}
\address{Department of Mathematics and Scientific Computing, University of Graz, \newline 8010 Graz, Austria}
\email{zahra.nazemian@uni-graz.at}
\email{doniyor.yazdonov@uni-graz.at}

\begin{abstract}
A right module is isoartinian (resp., isonoetherian) if every descending  (resp., ascending) chain of submodules is eventually
constant up to isomorphism, and a ring $R$ is right isoartinian (resp., isonoetherian) if $R_R$ is isoartinian (resp., isonoetherian). 
Facchini and Nazemian proved that a semiprime right isoartinian ring is right noetherian if
and only if it is right isonoetherian if  and only if it has finite Goldie dimension and that every commutative semiprime
isoartinian ring is noetherian. We construct, over an arbitrary field $K$, a (non commutative) prime (right and left)  isoartinian ring that is neither right nor left noetherian. 
\end{abstract}
\maketitle
\section{Introduction}
Throughout, rings are associative with identity and modules are unital right modules.
Following Facchini and Nazemian~\cite{FN}, a module $M_R$ is called \emph{isoartinian} if for
every descending chain
\[
M\supseteq M_1\supseteq M_2\supseteq\cdots
\]
of submodules of $M$ there is an index $n\ge 1$ such that $M_i\cong M_n$ for every $i\ge n$.
A ring $R$ is \emph{right isoartinian} if the module $R_R$ is isoartinian. Every right artinian
ring is right isoartinian, and a domain is right isoartinian if and only if it is a principal
right ideal domain~\cite{FN}.
Semiprime right isoartinian rings are studied in \cite[Section~4]{FN}. Such rings are right
non-singular and satisfy the ascending chain condition on annihilator ideals
\cite[Lemmas~4.3 and~4.5]{FN}. Moreover, by \cite[Theorem~4.6]{FN}, for a semiprime right
isoartinian ring $R$ the following are equivalent: $R$ is right noetherian; $R$ is right
Goldie; $R_R$ has a finite Goldie dimension; $R_R$ is a sum of the isosimple right ideals. The
semiprime right isoartinian right noetherian rings are exactly the finite direct products of
matrix rings over principal right ideal domains \cite[Theorem~4.7]{FN}. In the commutative case
the noetherian hypothesis is automatic: a commutative semiprime isoartinian ring is a finite
direct product of principal ideal domains \cite[Corollary~4.8]{FN}.
It is therefore natural to ask whether every semiprime right isoartinian ring is right
noetherian, that is, whether the finiteness of the Goldie dimension in \cite[Theorem~4.6]{FN}
is automatic, as it is for commutative rings. This question has been open for years and since the  definition of isonoetherian (isoartinian) rings in \cite{FN}, isoartinian rings and modules have been studied in various sources;  for example 
\cite{1, 2, 3, 4, 5, 6, 7, 8, 9}.  The purpose of this note is to show that the
answer is negative.
\begin{theorem}\label{thm:main}
Let $K$ be a field. There is a ring $R$, constructed in Section~\ref{sec:ring}, with the
following properties:
\begin{enumerate}[label=\textup{(\arabic*)}]
\item $R$ is a prime ring; in particular $R$ is semiprime;
\item $R$ is right and left isoartinian;
\item $R$ is neither right noetherian nor left noetherian.
\end{enumerate}
\end{theorem}

\section{The ring \texorpdfstring{$R$}{R}}\label{sec:ring}

Let \(K\) be a field, and let
$
V=K[[X]]
$
be the \(K\)-vector space of formal power series

$$
f(X)=\sum_{n=0}^{\infty}a_nX^n,
\qquad a_n\in K.
$$
The {\sf{backward shift}} is the \(K\)-linear operator
$$
D:K[[X]]\longrightarrow K[[X]]
$$
defined by
$$
D\left(\sum_{n=0}^{\infty}a_nX^n\right)
=
\sum_{n=0}^{\infty}a_{n+1}X^n.
$$
In particular,
$$
D(1)=0,
\qquad
D(X^n)=X^{n-1}\quad(n\geq1).
$$
For a formal power series $u(X)\in K[[X]],$
the {\sf {multiplication operator}} \(M_u\) is the \(K\)-linear operator
$$
M_u:K[[X]]\longrightarrow K[[X]]
$$
defined by
$$
M_u(f)=u f.
$$
Thus,

$$
M_u\left(\sum_{n=0}^{\infty}a_nX^n\right)
=
u(X)\sum_{n=0}^{\infty}a_nX^n.
$$
For \(N\geq0\), the operator $D ^N $ means the $N$-fold composition of $D$ and  
$D^N M_u$ means the composition
$$
D^N M_u(f)=D^N\bigl(M_u(f)\bigr)
=D^N(uf).
$$
A \(K\)-linear operator
$F:K[[X]]\longrightarrow K[[X]]$
is said to have {\sf {finite rank}} if its image is finite-dimensional over \(K\):
$$
\dim_K(\operatorname{im}F)<\infty.
$$
For example, the operator $P_0(f)=f(0)$ has rank one, since $\operatorname{im}P_0=K\cdot1.$

Two \(K\)-linear operators \(T,S:K[[X]]\to K[[X]]\) are said to {\sf {agree up to a finite-rank operator}} if
$
T-S
$ has finite rank. Equivalently, there exists a finite-rank operator \(F\) such that
$
T=S+F.$

In particular, an operator \(T\) agrees up to finite rank with \(D^NM_u\) if  $T-D^NM_u$ has finite rank, or equivalently,
$T=D^NM_u+F$, for some finite-rank operator \(F\).

Let 
$
\operatorname{End}_K(K[[X]])
$ denote the ring of all \(K\)-linear operators
$$
K[[X]]\longrightarrow K[[X]],
$$
where addition is defined pointwise and multiplication is composition:
$$
(S+T)(f)=S(f)+T(f),
\qquad
(ST)(f)=S(T(f)).
$$

The ideal  of $
\operatorname{End}_K(K[[X]])
$, which contains all operators of finite rank, is denoted by $\mathcal{F}$. 

\begin{definition}
Let
\[
R=\{\,r\in\E:\ r-\Dsh^N\Mu{u}\in\F \text{ for some } N\ge0,\ u\in\Kx\,\}.
\]
\end{definition}

$R$ contains $ \mathcal{F}$  and also 
 the backward shift $D$ and the operator $M_X$ of multiplication by $X$, which satisfy
$DM_X=1\neq M_XD$. We prove in Proposition \ref{prop:subring}, that $R$  is a ring and so a noncommutative one.  

Let $\Vd=\operatorname{Hom}_K(V,K)$ be the full algebraic dual of $V$. For $w\in V$ and
$\psi\in\Vd$ let $\tht{w}{\psi}\in\F$ be the operator
\[
\tht{w}{\psi}(x)=\psi(x)\,w\qquad(x\in V).
\]
Then $\im\tht{w}{\psi}\subseteq Kw$, and for every $r\in\E$ we have
\begin{equation}\label{eq:theta} \tag{E2.1}
r\,\tht{w}{\psi}=\tht{r(w)}{\psi},\qquad \tht{w}{\psi}\,r=\tht{w}{\psi\circ r}.
\end{equation}
The map    $V \times \Vd  \to \mathcal{F} $, defined by $(w,\psi)\mapsto\tht{w}{\psi}$ is $K$-bilinear. Every $f\in\F$ is a finite sum of
such operators with $w\in\im f$: if $w_1,\dots,w_n$ is a basis of $\im f$ and
$\chi_1,\dots,\chi_n\in\Vd$ are defined by $f(x)=\sum_i\chi_i(x)w_i$ (the coordinates of
$f(x)$ in this basis), then
\begin{equation}\tag{E2.2}\label{eq:decomp}
f=\sum_{i=1}^n\tht{w_i}{\chi_i}.
\end{equation}
For $p=\sum_{n\ge0}p_nX^n\in V$ let $\delta(p)=p_0$; thus $\delta\in\Vd$. Similarly, for every $ n$, we can define 
 $\delta_n \in \Vd$ such that  $\delta_n (p) = p_n$. In particular, for every nonzero $p \in V$, there exists $ \psi \in \Vd$ such that $ \psi ( p) \neq 0 $. 
Note that, for every $ u \in 
\Kx$,   $\Mu{X}$ is injective
because $\Kx$ is a domain. We have: 
\begin{lemma}\label{lem:shift}
Let $u\in\Kx$ and $N,k\ge0$. Then:
\begin{enumerate}[label=\textup{(\alph*)}]
\item $\Dsh\Mu{X}=1$ and $\Mu{X}\Dsh=1-\tht{1}{\delta}$; in particular $\Mu{X}\Dsh(q)=q$
whenever $ q \in \Kx$  and $\delta(q)=0$;
\item $\Dsh^k\Mu{X^k}=1$, and hence $\Dsh^N\Mu{u}=\Dsh^{N+k}\Mu{X^ku}$;
\item $\Dsh\Mu{u}-\Mu{u}\Dsh=\tht{\Dsh(u)}{\delta}$; in particular it has rank at most one;
\item $\Dsh^N\Mu{u}-\Mu{u}\Dsh^N\in\F$.
\end{enumerate}
\end{lemma}
\begin{proof}
(a) It is clear by the definitions. \\
(b) The first identity follows from (a) by induction on $k$, using the equality
$\Dsh^{k+1}\Mu{X^{k+1}}=\Dsh^k(\Dsh\Mu{X})\Mu{X^k}$. Then
$\Dsh^{N+k}\Mu{X^ku}=\Dsh^N(\Dsh^k\Mu{X^k})\Mu{u}=\Dsh^N\Mu{u}$. \\
(c) Let $p\in V$. By (a),
\begin{align*}
X\bigl(\Dsh(up)-u\Dsh(p)\bigr)&=\bigl(up-\delta(u)\delta(p)\bigr)-u\bigl(p-\delta(p)\bigr)\\
&=\delta(p)\bigl(u-\delta(u)\bigr)=X\,\delta(p)\Dsh(u).
\end{align*}
Since $X$ is not a zero divisor in $\Kx$, we get $\Dsh(up)-u\Dsh(p)=\delta(p)\Dsh(u) = \tht{\Dsh(u)}{\delta} (p)$. This proves (c). \\
(d) By induction on $N$, using (c), the identity
\[
\Dsh^{N+1}\Mu{u}-\Mu{u}\Dsh^{N+1}=\Dsh^N(\Dsh\Mu{u}-\Mu{u}\Dsh)+(\Dsh^N\Mu{u}-\Mu{u}\Dsh^N)\Dsh,
\]
and the fact that $\F$ is a two-sided ideal of $\E$.
\end{proof}

\begin{proposition}\label{prop:subring}
$R$ is a subring of $\E$ generated by  $\F$, $\Dsh$, and $\Mu{u}$ for every $u\in\Kx$. More
precisely, if  $r-\Dsh^N\Mu{u}\in\F$ and $s-\Dsh^{N'}\Mu{u'}\in\F$, then
\[
rs-\Dsh^{N+N'}\Mu{uu'}\in\F\quad\text{and}\quad
r+s-\Dsh^{N+N'}\Mu{X^{N'}u+X^Nu'}\in\F .
\]
\end{proposition}
\begin{proof}
Taking $u=0$ shows $\F\subseteq R$; taking $(N,u)=(1,1)$, respectively $(0,u)$, shows
$\Dsh\in R$, respectively $\Mu{u}\in R$; and $1=\Dsh^0\Mu{1}\in R$. Write
$f=r-\Dsh^N\Mu{u}$ and $g=s-\Dsh^{N'}\Mu{u'}$. Then
\[
\begin{aligned}
rs - \Dsh^{N+N'}\Mu{uu'}
&= \bigl(f + \Dsh^N\Mu{u}\bigr)s
   - \Dsh^{N+N'}\Mu{uu'} \\[2pt]
&= fs + \Dsh^N\Mu{u}\,s
   - \Dsh^{N+N'}\Mu{uu'} \\[2pt]
&= fs + \Dsh^N\Mu{u}
   \bigl(g + \Dsh^{N'}\Mu{u'}\bigr)
   - \Dsh^{N+N'}\Mu{uu'} \\[2pt]
&= fs + \Dsh^N\Mu{u}\,g
   + \Dsh^N
     \bigl(\Mu{u}\Dsh^{N'} - \Dsh^{N'}\Mu{u}\bigr)
     \Mu{u'}.
\end{aligned}
\]

which lies in $\F$ by Lemma~\ref{lem:shift}(d). By Lemma~\ref{lem:shift}(b),
\[
\Dsh^N\Mu{u}=\Dsh^{N+N'}\Mu{X^{N'}u}\quad\text{and}\quad\Dsh^{N'}\Mu{u'}=\Dsh^{N+N'}\Mu{X^{N}u'}.
\]
Using the fact that $ M_{a + b} = M_a + M_b$, for every $a, b \in \Kx$, we have that  $r+s-\Dsh^{N+N'}\Mu{X^{N'}u+X^Nu'}=f+g\in\F$. Finally
$-r-\Dsh^N\Mu{-u}=-f\in\F$. Hence $R$ is a subring of $\E$.
\end{proof}
\begin{lemma}\label{lem:inverse}
Let $r\in\E$, and suppose that $r-\Dsh^N\Mu{u}\in\F$ with $N\ge0$ and $0\neq u\in\Kx$. Then
there is $s\in R$ such that $sr - 1, rs-1\in\F$. In particular, for every $r\in R\setminus\F$ there
is $s\in R$ with $ sr - 1, 
rs-1\in\F$.
\end{lemma}
\begin{proof}
Since $u$ is nonzero, we can find integer number $k \geq 0$ and $W \in V$ with $\delta(w)\ne0$ such that  $u=X^kw$. Then $w$ is invertible in $\Kx$. Put
$s=\Dsh^k\Mu{w^{-1}X^N}\in R$. By Proposition~\ref{prop:subring} and
Lemma~\ref{lem:shift}(b),
\[
rs-1=rs-\Dsh^{N+k}\Mu{X^{N+k}}=rs-\Dsh^{N+k}\Mu{u w^{-1}X^N}\in\F.
\]

Also 
$sr-1=sr-D^{k+N}M_{w^{-1}X^Nu}\in\mathcal{F}$, because $uw^{-1}X^N=X^{N+k}$.
This proves the first part of the lemma. For particular part, note that 
if $r\in R\setminus\F$ and $r-\Dsh^N\Mu{u}\in\F$, then $u\neq0$, since otherwise $r\in\F$.
\end{proof}


\section{Primeness and failure of the noetherian conditions}\label{sec:prime}
\begin{proposition}\label{prop:prime}
For all nonzero $a,b\in R$ there is $t\in\F\subseteq R$ with $atb\ne0$. Hence $R$ is a prime
ring; in particular $R$ is semiprime, i.e., $R$ has no nonzero two-sided ideal $I$ with
$I^2=0$.
\end{proposition}
\begin{proof} Let $r, s \in R$ be nonzero. Then 
choose $p,q\in V$ with $a(p)\ne0$ and $b(q)\ne0$, and choose $\psi\in\Vd$ with
$\psi(b(q))\ne0$. For $t=\tht{p}{\psi}$ we get $atb(q)=\psi(b(q))\,a(p)\ne0$. If $I$ is a
two-sided ideal with $I^2=0$ and $0\ne a\in I$, then $ata\in I^2=0$ for all $t\in R$,
contradicting the first assertion.
\end{proof}

\begin{proposition}\label{prop:notnoeth}
The ideal $\F$ of $R$ is neither finitely generated as a right ideal nor finitely generated as
a left ideal. Consequently $R$ is neither right noetherian nor left noetherian.
\end{proposition}
\begin{proof}
Suppose $\F=f_1R+\dots+f_nR$ with $f_i\in\F$. Then every element of $\F$ has image contained in
the finite-dimensional subspace $U=\im f_1+\dots+\im f_n$. But for every $w\in V$ the
operator $\tht{w}{\delta}\in\F$ has $w=\tht{w}{\delta}(1)$ in its image, so $U=V$,
contradicting $\dim_K V=\infty$.
Suppose $\F=Rf_1+\dots+Rf_n$ with $f_i\in\F$. Then every element of $\F$ vanishes on
$L=\bigcap_i\ker f_i$. Since $V/L$ embeds into $\bigoplus_i\im f_i$, the subspace $L$ has finite
codimension, so $L\ne0$. Take $0\ne p\in L$ and $\psi\in\Vd$ with $\psi(p)\neq0$. Then
$\tht{1}{\psi}\in\F$ and $\tht{1}{\psi}(p)=\psi(p)\ne0$, a contradiction.
\end{proof}

\section{The right ideals of \texorpdfstring{$R$}{R}}\label{sec:rightideals}
Right ideals of $R$ are described in this section. We give a fully characterization of them in Theorem \ref{thm:classification}.

\begin{definition}
For a right ideal $A$ of $R$ put
\[
W(A)=\{\,w\in V:\ \tht{w}{\psi}\in A\ \text{for every}\ \psi\in\Vd\,\}.
\]
For a subspace $W$ of $V$ put
\[
\F_W=\{\,f\in\F:\ \im f\subseteq W\,\}.
\]
\end{definition}
By bilinearity of $\theta$, $W(A)$ is a $K$-subspace of $V$, and clearly $A\subseteq B$, where $B$ is also a right ideal of $R$, 
implies $W(A)\subseteq W(B)$. Since $\im(fr)\subseteq\im f$, $\F_W$ is a right ideal of $R$.
The following lemma shows that the right ideals of $R$ contained in $ \F$ are exactly in form of 
$ \F _ W $, for some subspace $ W$ of $V$. 
\begin{lemma}\label{lem:WA}
Let $A$ be a right ideal of $R$.
\begin{enumerate}[label=\textup{(\alph*)}]
\item If $a\in A$, then $\im a\subseteq W(A)$.
\item If $f\in\F$ and $\im f\subseteq W(A)$, then $f\in A$.
\item If $A\subseteq\F$, then $A=\F_{W(A)}$.
\end{enumerate}
\end{lemma}
\begin{proof}
(a) Let $ p\in V$ and $\psi\in\Vd$. By~\eqref{eq:theta},
$\tht{a(p)}{\psi} =a\,\tht{p}{\psi}\in aR\subseteq A$. Hence $a(p)\in W(A)$. \\
(b) Write $f=\sum_i\tht{w_i}{\chi_i}$ as in~\eqref{eq:decomp}, with $w_i\in\im f\subseteq W(A)$.
Each summand lies in $A$ by definition of $W(A)$, hence $f\in A$. \\
(c) follows from (a) and (b).
\end{proof}
\begin{lemma}\label{lem:finrank-iso}
Let $W,W'$ be subspaces of $V$ with $\dim_K W=\dim_K W'$ \textup{(}as cardinal
numbers\textup{)}. Then $\F_W\cong\F_{W'}$ as right $R$-modules.
\end{lemma}
\begin{proof}
Let $e\colon W\to W'$ be a $K$-linear isomorphism. Extend $e$ and $e^{-1}$ to $K$-linear maps
$\sigma,\tau\in\E$, so that $\sigma|_W=e$ and $\tau|_{W'}=e^{-1}$. For $f\in\F_W$ we have
$\sigma f\in\F$ and $\im(\sigma f)=e(\im f)\subseteq W'$, so $f\mapsto\sigma f$ is a map
$\F_W\to\F_{W'}$; it is a homomorphism of right $R$-modules because
$\sigma(fr)=(\sigma f)r$. Similarly $g\mapsto\tau g$ is a homomorphism $\F_{W'}\to\F_W$. For
$f\in\F_W$ we have $\tau\sigma f=f$, since $\tau\sigma$ is the identity on $W\supseteq\im f$;
similarly $\sigma\tau g=g$ for $g\in\F_{W'}$. Hence the two maps are mutually inverse.
\end{proof}
\begin{lemma}\label{lem:infinite}
Let $A$ be a right ideal of $R$ with $A\not\subseteq\F$. Then:
\begin{enumerate}[label=\textup{(\alph*)}]
\item $A=\{\,r\in R:\ \im r\subseteq W(A)\,\}$;
\item there is $g\in\F$ with $\ker g\subseteq W(A)$.
\end{enumerate}
\end{lemma}
\begin{proof}
Let $a\in A\setminus\F$. By Lemma~\ref{lem:inverse} there is $s\in R$ with $g:=as-1\in\F$.
Put $e=as=1+g\in A$. \\
(a) One inclusion is Lemma~\ref{lem:WA}(a). Conversely, let $r\in R$ with
$\im r\subseteq W(A)$. Then $er\in A$, so $\im(er)\subseteq W(A)$ by Lemma~\ref{lem:WA}(a).
Hence $gr=er-r\in\F$ has image contained in $W(A)$, and $gr\in A$ by Lemma~\ref{lem:WA}(b).
Therefore $r=er-gr\in A$. \\
(b) If $g(p)=0$, for some $p \in V$, then $p=e(p)\in\im e\subseteq W(A)$ by Lemma~\ref{lem:WA}(a).
\end{proof}
The key technical step is the following.
\begin{definition}
A subspace $W$ of $V$ is \emph{good} if there are $u,v\in R$ such that $u$ is injective,
$\im u=W$, and $u(v(w))=w$ for every $w\in W$. We then say that $u,v$ \emph{witness} that $W$
is good.
\end{definition}
\begin{proposition}\label{prop:good}
Let $f\in\F$. Every subspace $W$ of $V$ with $\ker f\subseteq W$ is good.
\end{proposition}
We prove Proposition~\ref{prop:good} in several steps. Note that $V$ itself is good (take
$u=v=1$).
\begin{lemma}\label{lem:good-hyperplane}
If $0\ne\varphi\in\Vd$, then $\ker\varphi$ is good.
\end{lemma}
\begin{proof}
First we find $a\in V$ with $\delta(a)\ne0$ and $\varphi(a)\ne0$. Choose $a_1$ with
$\varphi(a_1)\ne0$. If $\delta(a_1)\ne0$, take $a=a_1$. If $\delta(a_1)=0$, take $a=1$ when
$\varphi(1)\neq0$, and $a=a_1+1$ when $\varphi(1)=0$; in both cases $\delta(a)=1$ and
$\varphi(a)\neq0$.
Put $\alpha=\delta(a)^{-1}\delta\in\Vd$ and $\beta=\varphi(a)^{-1}\varphi\in\Vd$, so that
$\alpha(a)=\beta(a)=1$, $\ker\alpha=\ker\delta$ and $\ker\beta=\ker\varphi$. \\ Define
\[
g=1+\tht{a}{\alpha-\beta},\qquad \pi=1-\tht{a}{\alpha},\qquad u=g\Mu{X},\qquad v=\Dsh\pi .
\]
By Proposition~\ref{prop:subring}, $g,\pi,u,v\in R$. We check that $u,v$ witness that
$\ker\varphi$ is good. \\
\emph{$u$ is injective:} Since $M_X$ is injective, it is enough to show  that $g$  is injective. 
If $g(p)=0$, for some $ p \in V$, then $p=-c\,a$ with $c=\alpha(p)-\beta(p)$. Applying
$\alpha$ and $\beta$ gives $\alpha(p)=-c=\beta(p)$. Since 
 $c=\alpha(p)-\beta(p)$, we have that  $c=0$ and $p=0$.  That is $g$ and consequently $u$ is injective. \\
 
\emph{$\im u = \ker\varphi$: } Let $p\in V$ and $q=Xp$. Then $\alpha(q)=0$, so
$g(q)=q-\beta(q)a$ and $\beta(g(q))=\beta(q)-\beta(q)=0$; thus
$g(q) $ is  the kernel of $ \beta$ which is equal to the kernel  of $\varphi$. 
Thus 
$\varphi(u(p))=0$ and consequently  $\im u \subseteq \ker\varphi$. To show the equaality we prove that 
{$u(v(h))=h$ for $h\in\ker\varphi$.} Let $h\in\ker\varphi$. We have
$\pi(h)=h-\alpha(h)a$, so $\alpha(\pi(h))=0$, i.e., $\delta(\pi(h))=0$, and
$\beta(\pi(h))=-\alpha(h)$. By Lemma~\ref{lem:shift}(a), $\Mu{X}\Dsh\pi(h)=\pi(h)$, and therefore
\[
u(v(h))=g(\pi(h))=\pi(h)+\bigl(\alpha(\pi(h))-\beta(\pi(h))\bigr)a=h-\alpha(h)a+\alpha(h)a=h.
\]
In particular $\ker\varphi\subseteq\im u$, so $\im u=\ker\varphi$.
\end{proof}
\begin{lemma}\label{lem:good-codim1}
If $W$ is a subspace of $V$ and $c\in V$ satisfies $W+Kc=V$, then $W$ is good.
\end{lemma}
\begin{proof}
If $c\in W$, then $W=V$ is good. Otherwise there is $\varphi\in\Vd$ with $\varphi(W)=0$ and
$\varphi(c)=1$, and then $\ker\varphi=W$ because $W+Kc=V$. Apply
Lemma~\ref{lem:good-hyperplane}.
\end{proof}
\begin{lemma}\label{lem:good-comp}
Let $W\subseteq W'$ be subspaces of $V$, and suppose that $u',v'\in R$ witness that $W'$ is
good. If the subspace $u'^{-1}(W)=\{x\in V: u'(x)\in W\}$ is good, then $W$ is good.
\end{lemma}
\begin{proof}
Let $u_0,v_0\in R$ witness that $u'^{-1}(W)$ is good, and put $u=u'u_0$, $v=v_0v'$. Then $u$
is injective and
$\im u=u'\bigl(u'^{-1}(W)\bigr) \subseteq W$. If $w\in W$, then
$u'(v'(w))=w\in W$, so $v'(w)\in u'^{-1}(W)$, and hence
$u(v(w))=u'\bigl(u_0v_0(v'(w))\bigr)=u'(v'(w))=w$.
\end{proof}
\begin{lemma}\label{lem:good-fincodim}
If $W$ is a subspace of $V$ and $c_1,\dots,c_n\in V$ satisfy
$W+Kc_1+\dots+Kc_n=V$, then $W$ is good.
\end{lemma}
\begin{proof}
Induction on $n$, the case $n=0$ being trivial.
Case $n = 1$ is Lemma \ref{lem:good-codim1}. Let $n\ge2$ and put $W'=W+Kc_1$. Since
$W'+Kc_2+\dots+Kc_n=V$, the subspace $W'$ is good by induction, witnessed by $u',v'$, say. We
claim that $u'^{-1}(W)+K\,v'(c_1)=V$. Let $x\in V$. Since $u'(x)\in W'$, we can write
$u'(x)=w+tc_1$ with $w\in W$ and $t\in K$. As $c_1\in W'$, we have $u'(v'(c_1))=c_1$, whence
$u'(x-t\,v'(c_1))=w\in W$, i.e., $x-t\,v'(c_1)\in u'^{-1}(W)$. This proves the claim. By
Lemma~\ref{lem:good-codim1}, $u'^{-1}(W)$ is good, and by Lemma~\ref{lem:good-comp}, $W$ is
good.
\end{proof}
\begin{proof}[Proof of Proposition~\ref{prop:good}]
Let $C$ be a complement of $\ker f$ in $V$. Then $C\cong V/\ker f\cong\im f$ is
finite-dimensional; let $c_1,\dots,c_n$ be a basis of $C$. Then
$W+Kc_1+\dots+Kc_n\supseteq\ker f+C=V$, and Lemma~\ref{lem:good-fincodim} applies.
\end{proof}
\begin{lemma}\label{lem:iso-R}
Let $W$ be a good subspace of $V$, and let $A$ be a right ideal of $R$ such that
$A=\{\,r\in R:\im r\subseteq W\,\}$. Then $A\cong R_R$.
\end{lemma}
\begin{proof}
Let $u,v\in R$ witness that $W$ is good. For $r\in R$ we have $\im(ur)\subseteq\im u=W$, so
$ur\in A$, and $\lambda\colon R\to A$, $\lambda(r)=ur$, is a homomorphism of right
$R$-modules. It is injective because $u$ is injective. It is surjective: if $a\in A$, then
$va\in R$ and $u(v(a(x)))=a(x)$ for all $x\in V$ since $a(x)\in W$; thus $\lambda(va)=a$.
\end{proof}
Combining the preceding results we obtain a classification of the right ideals of $R$ up to
isomorphism.
\begin{theorem}\label{thm:classification}
Let $A$ be a right ideal of $R$.
\begin{enumerate}[label=\textup{(\alph*)}]
\item If $A\not\subseteq\F$, then $A\cong R_R$.
\item If $A\subseteq\F$, then $A=\F_{W(A)}$. If $B\subseteq \F$ is another right ideal and
$\dim_K W(A)=\dim_K W(B)$, then $A\cong B$.
\end{enumerate}
\end{theorem}
\begin{proof}
(a) By Lemma~\ref{lem:infinite}, $A=\{r\in R:\im r\subseteq W(A)\}$ and $W(A)$ contains the
kernel of some $g\in\F$. By Proposition~\ref{prop:good}, $W(A)$ is good, and
Lemma~\ref{lem:iso-R} gives $A\cong R_R$. \\
(b) This follows from Lemma~\ref{lem:WA}(c) and Lemma~\ref{lem:finrank-iso}.
\end{proof}

\section{The left ideals of \texorpdfstring{$R$}{R}}\label{sec:rightideals}
Left ideals  of $R$ are described in this section. 
 For a left ideal $L$ of $R$ put
\[
U(L)=\{\,\psi\in V^{*}:\ \theta_{w,\psi}\in L\ \text{for every}\ w\in V\,\}.
\]
This is a subspace of $V^{*}$, and $L\subseteq L'$ implies $U(L)\subseteq U(L')$.

\begin{lemma}\label{iso:lem:UL}
Let $L$ be a left ideal of $R$.
\begin{enumerate}
\renewcommand{\labelenumi}{\textup{(\alph{enumi})}}
\item If $a\in L$ and $\varphi\in V^{*}$, then $\varphi\circ a\in U(L)$.
\item If $f\in\mathcal{F}$ and $\varphi\circ f\in U(L)$ for every $\varphi\in V^{*}$, then
$f\in L$.
\end{enumerate}
\end{lemma}

\begin{proof}
(a) $\theta_{w,\varphi\circ a}=\theta_{w,\varphi}\,a\in L$ by~\eqref{eq:theta}. \\

(b) If $f=0$ then $f\in L$, so assume $f\neq0$. Since $f\in\mathcal{F}$, the subspace
$\operatorname{im} f$ is finite-dimensional; let $w_1,\dots,w_n$ be a $K$-basis of it.

\emph{Step 1: coordinate functionals.} Choose a $K$-subspace $C$ of $V$ with
$V=\operatorname{im} f\oplus C$ (extend $w_1,\dots,w_n$ to a basis of $V$). For
$i=1,\dots,n$ define $\varphi_i\in V^{*}$ by
\[
\varphi_i\Bigl(\sum_{j=1}^n c_jw_j+c\Bigr)=c_i\qquad(c_1,\dots,c_n\in K,\ c\in C).
\]
This is well defined and $K$-linear because the sum $\operatorname{im} f\oplus C$ is direct and
$w_1,\dots,w_n$ are linearly independent; and $\varphi_i(w_j)=\delta_{ij}$.

\emph{Step 2: decomposition of $f$.} Let $x\in V$. Since $f(x)\in\operatorname{im} f$, we can
write $f(x)=\sum_{j=1}^n c_jw_j$ with $c_j\in K$. Applying $\varphi_i$ gives
$\varphi_i(f(x))=\sum_jc_j\delta_{ij}=c_i$. Hence
\[
f(x)=\sum_{i=1}^n\varphi_i(f(x))\,w_i=\sum_{i=1}^n(\varphi_i\circ f)(x)\,w_i
=\sum_{i=1}^n\theta_{w_i,\varphi_i\circ f}(x).
\]
As this holds for every $x\in V$, we obtain the equality of operators
\[
f=\sum_{i=1}^n\theta_{w_i,\psi_i},\qquad\text{where }\psi_i:=\varphi_i\circ f\in V^{*};
\]
this is~\eqref{eq:decomp}.

\emph{Step 3: each summand lies in $L$.} Fix $i$. Applying the hypothesis with
$\varphi=\varphi_i$ gives $\psi_i=\varphi_i\circ f\in U(L)$. By the definition of $U(L)$,
$\theta_{w,\psi_i}\in L$ for \emph{every} $w\in V$; in particular, taking $w=w_i$,
$\theta_{w_i,\psi_i}\in L$.

\emph{Step 4: conclusion.} $L$ is closed under addition, so by Steps 2 and 3,
$f=\sum_{i=1}^n\theta_{w_i,\psi_i}\in L$.

\end{proof}

\begin{lemma}\label{iso:lem:finrank-left}
Let $L\subseteq\mathcal{F}$ be a left ideal of $R$ and let $(\psi_j)_{j\in J}$ be a $K$-basis of
$U(L)$. Then
\[
\Sigma\colon V^{(J)}\to L,\qquad (v_j)_{j\in J}\mapsto\sum_{j\in J}\theta_{v_j,\psi_j},
\]
is an isomorphism of left $R$-modules. Consequently, if $L,L'\subseteq\mathcal{F}$ are left ideals
with $\dim_KU(L)=\dim_KU(L')$, then $L\cong L'$.
\end{lemma}

\begin{proof}
Throughout, an element $(v_j)_{j\in J}$ of $V^{(J)}$ has only finitely many nonzero entries, so the
sum defining $\Sigma$ is a finite sum of operators. For $j\in J$ and $w\in V$ write
$\iota_j(w)\in V^{(J)}$ for the family with $w$ in position $j$ and $0$ elsewhere; then
$\Sigma(\iota_j(w))=\theta_{w,\psi_j}$.

\emph{Step 1: $\Sigma$ is a well-defined homomorphism of left $R$-modules into $L$.} Each $\psi_j$
lies in $U(L)$, so by the definition of $U(L)$ every $\theta_{v_j,\psi_j}$ lies in $L$. A finite
sum of elements of $L$ is in $L$, so $\Sigma((v_j))\in L$. Additivity follows from the
bilinearity of $(w,\psi)\mapsto\theta_{w,\psi}$. For $r\in R$, using~\eqref{eq:theta},
\[
r\,\Sigma((v_j))=\sum_j r\,\theta_{v_j,\psi_j}=\sum_j\theta_{r(v_j),\psi_j}=\Sigma((r\cdot v_j)).
\]

\emph{Step 2: a key identity.} Let $\chi\in V^{*}$, $v\in V$ and $\psi\in V^{*}$. Then the
functional $\chi\circ\theta_{v,\psi}\in V^{*}$ equals $\chi(v)\,\psi$. Indeed, for every $x\in V$,
\[
(\chi\circ\theta_{v,\psi})(x)=\chi\bigl(\psi(x)\,v\bigr)=\psi(x)\,\chi(v)
=\bigl(\chi(v)\,\psi\bigr)(x),
\]
using that $\chi$ is $K$-linear and $\psi(x)\in K$.

\emph{Step 3: $\Sigma$ is injective.} Let $(v_j)\in V^{(J)}$ with $\Sigma((v_j))=0$, that is,
$\sum_j\theta_{v_j,\psi_j}=0$ as an operator $V\to V$. Suppose, for a contradiction, that
$v_{j_0}\neq0$ for some $j_0\in J$. Choose $\chi\in V^{*}$ with $1 = \chi(v_{j_0})\neq0$. Such $\chi$
exists: extend $\{v_{j_0}\}$ to a $K$-basis of $V$ and let $\chi$ be the coordinate functional
of $v_{j_0}$, so $\chi(v_{j_0})=1$. Now compose the zero operator $\sum_j\theta_{v_j,\psi_j}$
on the left with $\chi$; the result is the zero functional. Since composition with $\chi$ is
additive and the sum is finite, Step~2 gives
\[
0=\chi\circ\Bigl(\sum_{j}\theta_{v_j,\psi_j}\Bigr)=\sum_{j}\chi\circ\theta_{v_j,\psi_j}
=\sum_{j}\chi(v_j)\,\psi_j\qquad\text{in }V^{*}.
\]
This is a finite linear relation among the $\psi_j$ with scalar coefficients $\chi(v_j)\in K$
(only finitely many are nonzero, because only finitely many $v_j$ are nonzero). Since
$(\psi_j)_{j\in J}$ is a basis of $U(L)$, the $\psi_j$ are linearly independent in $V^{*}$, so
every coefficient vanishes: $\chi(v_j)=0$ for all $j$. In particular $\chi(v_{j_0})=0$,
contradicting the choice of $\chi$. Hence $v_j=0$ for all $j$, and $\Sigma$ is injective.

\emph{Step 4: $\Sigma$ is surjective.} Let $L'=\Sigma(V^{(J)})$. By Step~1, $L'$ is a left
$R$-submodule of $L$, i.e.\ a left ideal of $R$ with $L'\subseteq L$, and therefore
$U(L')\subseteq U(L)$. Conversely, for each $j$ and every $w\in V$ we have
$\theta_{w,\psi_j}=\Sigma(\iota_j(w))\in L'$, so $\psi_j\in U(L')$. Since $U(L')$ is a subspace
containing the basis $(\psi_j)$ of $U(L)$, we get $U(L)\subseteq U(L')$, hence $U(L')=U(L)$.
Now let $f\in L$. Then $f\in\mathcal{F}$ because $L\subseteq\mathcal{F}$, and for every
$\varphi\in V^{*}$, Lemma~\ref{iso:lem:UL}(a) applied to $L$ gives
$\varphi\circ f\in U(L)=U(L')$. Lemma~\ref{iso:lem:UL}(b), applied to the left ideal $L'$, gives
$f\in L'$. Thus $L=L'$, and $\Sigma$ is surjective.

\emph{Step 5: the last assertion.} Let $(\psi_j)_{j\in J}$ and $(\psi'_k)_{k\in J'}$ be bases of
$U(L)$ and $U(L')$. By the first part, $L\cong V^{(J)}$ and $L'\cong V^{(J')}$ as left
$R$-modules. Since $|J|=\dim_KU(L)=\dim_KU(L')=|J'|$, there is a bijection $\sigma\colon J\to J'$,
and $(v_j)_{j\in J}\mapsto(v_{\sigma^{-1}(k)})_{k\in J'}$ is an isomorphism
$V^{(J)}\cong V^{(J')}$ of left $R$-modules. Hence $L\cong L'$.
\end{proof}


\begin{lemma}\label{iso:lem:kerdim}
Let $L$ be a left ideal of $R$ with $L\not\subseteq\mathcal{F}$.
\begin{enumerate}
\renewcommand{\labelenumi}{\textup{(\alph{enumi})}}
\item $L$ contains an element with finite-dimensional kernel.
\item If $x\in L$ and $l\in L$ satisfy $l(v)\neq0$ for some $v\in\ker x$, then there is
$x'\in L$ with $\ker x'\subsetneq\ker x$.
\item There is $x_0\in L$ with $\dim_K\ker x_0<\infty$ such that every element of $L$ vanishes on
$\ker x_0$.
\end{enumerate}
\end{lemma}

\begin{proof}
(a) Let $a\in L\setminus\mathcal{F}$. By Lemma~\ref{lem:inverse} there is $s\in R$ with
$g:=sa-1\in\mathcal{F}$. Then $sa\in L$, and if $sa(v)=0$ then $v=-g(v)$; so
$\ker(sa)\subseteq\operatorname{im} g$ is finite-dimensional.

(b) Pick $\varphi\in V^{*}$ with $\varphi(l(v))\ne0$ and put
$x'=M_Xx+\theta_{1,\varphi}\,l\in L$. If $x'(p)=0$, then $X\,x(p)+\varphi(l(p))\cdot1=0$;
comparing constant terms gives $\varphi(l(p))=0$, hence $Xx(p)=0$ and $x(p)=0$. So
$\ker x'\subseteq\ker x$. Moreover $x'(v)=\varphi(l(v))\cdot1\neq0$, so $v\notin\ker x'$.

(c) By (a) we may choose $x_0\in L$ with finite-dimensional kernel of minimal dimension. By (b),
every $l\in L$ vanishes on $\ker x_0$.
\end{proof}

\begin{lemma}\label{iso:lem:factor}
Let $x_0\in R$ have finite-dimensional kernel. If $r\in R$ vanishes on $\ker x_0$, then
$r\in Rx_0$.
\end{lemma}

\begin{proof}
Since $V/\ker x_0\cong\operatorname{im} x_0$ and $\dim_KV=\infty$, $x_0\notin\mathcal{F}$. By
Lemma~\ref{lem:inverse} there is $s\in R$ with $g:=sx_0-1\in\mathcal{F}$. Then
$f:=rg\in\mathcal{F}$ and $f$ vanishes on $\ker x_0$. Hence there is a well defined linear map
$\operatorname{im} x_0\to\operatorname{im} f$, $x_0(p)\mapsto f(p)$; extend it to some
$h\in\operatorname{End}_K(V)$ with $\operatorname{im} h\subseteq\operatorname{im} f$. Then
$h\in\mathcal{F}\subseteq R$, $f=hx_0$, and $r=rsx_0-rg=(rs-h)x_0\in Rx_0$.
\end{proof}

\begin{lemma}\label{iso:lem:kernels}
For every finite-dimensional subspace $W$ of $V$ there are $x,y\in R$ with $xy=1$ and
$\ker x=W$.
\end{lemma}

\begin{proof}
First let $W=Kc$ with $c\ne0 \in V $. There is $\psi\in V^{*}$ with $\psi(c)=1$ and $\psi(1)\ne0$. 
 Put
$\mu=\psi(1)$ and $a=1-c$, so $\psi(a)=\mu-1$, and
\[
g=1+\theta_{a,\psi},\qquad g'=1-\mu^{-1}\theta_{a,\psi},\qquad x=Dg,\qquad y=g'M_X .
\]
These lie in $R$. A direct computation gives
$gg'(v)=v+\psi(v)\bigl(1-\mu^{-1}-\mu^{-1}(\mu-1)\bigr)a=v$ and
$g'g(v)=v+\psi(v)\bigl(1-\mu^{-1}-\mu^{-1}(\mu-1)\bigr)a=v$, so $g'=g^{-1}$, and $g(c)=c+a=1$.
Hence $xy=Dgg'M_X=DM_X=1$, and
$\ker x=g^{-1}(\ker D)=g^{-1}(K\cdot1)=Kc$.

In general we induct on $\dim_KW$, the case $W=0$ being trivial ($x=y=1$). Write
$W=W_1+Kc$ with $W_1=\ker x_1$, $x_1y_1=1$, $x_1,y_1\in R$. By the first part there are
$x_2,y_2\in R$ with $x_2y_2=1$ and $\ker x_2=Kx_1(c)$ (if $x_1(c)=0$ take $x_2=y_2=1$). Then
$x=x_2x_1$ and $y=y_1y_2$ satisfy $xy=x_2x_1y_1y_2=1$ and
$\ker x=x_1^{-1}(Kx_1(c))=Kc+\ker x_1=W$.
\end{proof}

\begin{theorem}\label{iso:thm:left}
Let $L$ be a left ideal of $R$.
\begin{enumerate}
\renewcommand{\labelenumi}{\textup{(\alph{enumi})}}
\item If $L\not\subseteq\mathcal{F}$, then $L=Rx_0=\{r\in R:\ r(\ker x_0)=0\}$ for some
$x_0\in L$ with finite-dimensional kernel, and $L\cong{}_RR$.
\item If $L,L'\subseteq\mathcal{F}$ are left ideals with $\dim_KU(L)=\dim_KU(L')$, then
$L\cong L'$.
\end{enumerate}
\end{theorem}

\begin{proof}
(a) Let $x_0$ be as in Lemma~\ref{iso:lem:kerdim}(c) and $W=\ker x_0$. Then
$Rx_0\subseteq L\subseteq\{r\in R: r(W)=0\}\subseteq Rx_0$, the last inclusion by
Lemma~\ref{iso:lem:factor}. By Lemma~\ref{iso:lem:kernels} there are $x,y\in R$ with $xy=1$ and
$\ker x=W$. The map $\rho\colon R\to L$, $\rho(r)=rx$, is a homomorphism of left $R$-modules
with values in $L$, since $rx$ vanishes on $W$. It is injective, since $rx=0$ implies
$r=rxy=0$. It is surjective: if $l\in L$ and $v\in V$, then $v-yx(v)\in\ker x=W$, so
$l(v)=l(yx(v))$; thus $l=(ly)x=\rho(ly)$. \\

(b) is Lemma~\ref{iso:lem:finrank-left}.
\end{proof}

\section{Proof of the main theorem}\label{sec:main}
\begin{theorem}\label{thm:isoartinian}
The ring $R$ is right and left  isoartinian.
\end{theorem}
\begin{proof}
Let $A_1\supseteq A_2\supseteq\cdots$ be a descending chain of right ideals of $R$. \\
\emph{Case 1: $A_n\not\subseteq\F$ for every $n$.} Then $A_n\cong R_R$ for all $n$ by
Theorem~\ref{thm:classification}(a), and so $A_i\cong A_1$ for all $i\geq 1$. \\
\emph{Case 2: $A_{n_0}\subseteq\F$ for some $n_0$.} Then $A_i\subseteq\F$ for all $i\ge n_0$.
The cardinal numbers $\kappa_i=\dim_K W(A_i)$, $i\ge n_0$, form a non-increasing sequence,
because $W(A_{i+1})\subseteq W(A_i)$. Since the cardinal numbers are well ordered, there is
$m\ge n_0$ with $\kappa_m\le\kappa_i$ for all $i\ge n_0$. For $i\ge m$ we then have
$\kappa_m\le\kappa_i\le\kappa_m$, so $\kappa_i=\kappa_m$, and
Theorem~\ref{thm:classification}(b) yields $A_i\cong A_m$.
In both cases the chain is eventually constant up to isomorphism.

The same argument, with Theorem~\ref{iso:thm:left} in place of Theorem~\ref{thm:classification} and
the cardinals $\dim_KU(L_i)$ in place of $\dim_KW(A_i)$, applies to a descending chain
$L_1\supseteq L_2\supseteq\cdots$ of left ideals.
\end{proof}
\begin{proof}[Proof of Theorem~\ref{thm:main}]
Statement (1) is Proposition~\ref{prop:prime}, (2) is Theorem~\ref{thm:isoartinian}, and (3)
is Proposition~\ref{prop:notnoeth}.
\end{proof}
\begin{remark}\label{rem:goldie}
By \cite[Theorem~4.6]{FN}, a semiprime right isoartinian ring that is not right noetherian must
have infinite Goldie dimension (for Goldie dimension see, e.g., \cite[Section~2.2]{MR}). This
can be seen directly in our example. For $0\neq w\in V$ the right ideal
$\F_{Kw}=\{\tht{w}{\psi}:\psi\in\Vd\}$ is a minimal right ideal: if $\psi\neq0$, pick $x$ with
$\psi(x)=1$; then $\tht{w}{\psi}\,\tht{x}{\chi}=\tht{w}{\chi}$ for every $\chi\in\Vd$
by~\eqref{eq:theta}. By Theorem~\ref{thm:classification}(b), all the right ideals $\F_{Kw}$ are
isomorphic. If $w_1,w_2,\dots$ are linearly independent (for instance $w_n=X^n$), the sum
$\sum_n\F_{Kw_n}$ is direct, because an element of $\F_{Kw_n}\cap\sum_{j\ne n}\F_{Kw_j}$ has
image in $Kw_n\cap\sum_{j\neq n}Kw_j=0$. Moreover, by~\eqref{eq:decomp}, $\F$ is the sum of
the minimal right ideals $\F_{Kw}$, and every minimal right ideal $S$ of $R$ is contained in
$\F$, since otherwise $S\cong R_R$ by Theorem~\ref{thm:classification}(a), whereas $R_R$ is
not simple ($\F$ is a nonzero proper right ideal). Hence $\F$ is the right socle of $R$, and
it is an infinite direct sum of pairwise isomorphic simple right ideals.
\end{remark}

\subsection*{Acknowledgments}
This research was funded in whole or in part by the Austrian Science Fund (FWF) [Grant DOI 10.55776/P36742].

\section*{Declaration of AI assistance}

During the preparation of this paper, the authors used ChatGPT to refine the language, improve the presentation, and assist with the preparation of \LaTeX{} text. This work was carried out through face-to-face discussions at the University of Graz. The authors take full responsibility for the mathematical correctness and final content of the paper.


\end{document}